\documentclass[12pt,reqno]{amsart}
\usepackage{amsmath,amsthm,amscd,amsfonts,amssymb,color}
\usepackage{cite}
\usepackage[mathscr]{eucal}
\usepackage[bookmarksnumbered,colorlinks,plainpages]{hyperref}
\newtheorem{theorem}{Theorem}[section]

\newtheorem{open question}[theorem]{Open Question}
\theoremstyle{definition}
\newtheorem{definition}[theorem]{Definition}
\newtheorem{example}[theorem]{Example}
\newtheorem{remark}[theorem]{Remark}
\numberwithin{equation}{section}
\def\DJ{\leavevmode\setbox0=\hbox{D}\kern0pt\rlap
 {\kern.04em\raise.188\ht0\hbox{-}}D}

\begin{document}
\title[On some enriched contractions in Banach spaces]{On some enriched contractions in Banach spaces}

\author[P.\ Mondal, H.\ Garai, L.K. \ Dey]
{Pratikshan Mondal$^{1}$, Hiranmoy Garai$^{2}$, Lakshmi Kanta Dey$^{3}$.}

\address{{$^{1}$\,} Department of Mathematics,                                         					\newline \indent Durgapur Government College, 
                    Durgapur, India.}
                    \email{real.analysis77@gmail.com}
\address{{$^{2}$\,} Department of Mathematics,
                    \newline \indent National Institute of Technology
                    Durgapur, India.}
                    \email{hiran.garai24@gmail.com}
\address{{$^{3}$\,} Department of Mathematics,
                   \newline \indent National Institute of Technology
                    Durgapur, India.}
                    \email{lakshmikdey@yahoo.co.in}
\subjclass[2010]{$47$H$05$, $47$H$10$, $54$H$25$.}
\begin{abstract}
In this paper, we introduce two new types of enriched contractions, viz., enriched $\mathcal{A}$-contraction and enriched $\mathcal{A}'$-contraction. Then we obtain fixed points of mappings satisfying such contractions using the fixed point property of the average operator of the mappings. Further, we study the well-posedness and limit shadowing property of the fixed point problem involving the contractions, and give some examples to validate the results proved. We frame an open question related to the existence of a fixed point of such contractions. We also show that Berinde and  P\u{a}curar's recent results on different kinds  enriched contractions and  some well known classical fixed point results are particular cases of our results.

\textbf{Keywords:} Banach spaces; enriched $\mathcal{A}$-contraction; enriched $\mathcal{A}'$-contraction; well-posedness; limit shadowing property.

\end{abstract}

\maketitle

\section{Introduction and preliminaries}\label{sec:1}
Fixed point theory is one of the most rapidly growing branch of mathematics due to its applicability in different areas of pure and applied mathematics. Several researchers are engaged in this field to study fixed point of a self-map by generalizing the underlying  space as well as the contraction condition. One of the main focus of this study is to find a set of sufficient condition(s) to ensure the existence of a fixed point which is the self image of the map under consideration.  To reach such conditions, several mathematicians have introduced a number of contractions  and used these in several structures to obtain fixed point, common fixed point, coincidence point of a map. Some of such contractions are introduced recently and looks very interesting to study. One such contraction is enriched contraction, which is introduced very recently by Berinde and  P\u{a}curar, see \cite{BP20}.  
We first recall the definition of enriched contraction.
\begin{definition} {\bf (\cite[p. 2, Definition 2.1]{BP20}).}\label{d1} 
Let $(X, \|\cdot\|)$ be a normed linear space. A self map $T$ on $X$ is called an enriched contraction if there exist $b\in [0, \infty)$ and $\theta\in [0, b+1)$ such that
$$\|b(u-v)+Tu-Tv\|\le \theta \|u-v\|$$
for all $u, v\in X$. In this case the mapping $T$ is called a $(b, \theta)$-enriched contraction. The class of enriched contractions contains the Picard-Banach contractions as well as some non-expensive mappings, e.g., a Picard-Banach mapping with constant $k$ is a $(0, k)$-enriched contraction.
\end{definition}
Berinde and  P\u{a}curar showed that an enriched contraction  $T$ defined on a Banach space possesses a unique fixed point. They obtained this result by using the fixed point property of the map $T_\lambda$, where $T_\lambda(u)=(1-\lambda)u+\lambda Tu$, $0<\lambda<1$.

It is to be noted that the enriched contraction introduced by Berinde and  P\u{a}curar involves the displacement $\|u-v\|$ only, which indicates that this contraction is basically the enriched version of  Banach contraction. We know that for two points $u,v$ the other displacements are $\|u-Tu\|,\ \|v-Tv\|,\ \|u-Tv\|,\ \|v-Tu\|$ and there are a lot of interesting contractions involving these displacements, some of them are due to Kannan \cite{Rk68}, Chatterjea \cite{Ch}, Reich \cite{R3}, $\acute{C}$iri$\acute{c}$ \cite{C1,C7}, Bianchini \cite{B10}, Khan \cite{K78} etc. (see \cite{R4} for more contractions of this type). So it will be attractive to apply enrichment technique to the contractions by the above mentioned mathematicians. Berinde and  P\u{a}curar did this task for Kannan and Chatterjea mappings, see \cite{BP1,BP2}. So it is open in the literature to apply  enrichment technique to the contractions of Reich \cite{R3}, $\acute{C}$iri$\acute{c}$ \cite{C1,C7}, Bianchini \cite{B10}, Khan \cite{K78} etc. 
If we want to apply this technique to the above-named contractions separately, then we  have to prove a handful number of results. So if we can establish a very few results involving enriched contractions, from which all the enriched versions of the aforesaid contractions can be deduced, then it will be simple and concise to the literature.  
Motivated by this observation, in the present article, we introduce two new types of enriched contractions, viz., enriched $\mathcal{A}$-contraction and enriched $\mathcal{A}'$-contraction. Then we prove two fixed point results showing that a self mapping  $T$ satisfying anyone of these two contractions admits a fixed point if the underlying space is a Banach space. Further, we show that from our obtained results, the enriched versions of the contractions due to Kannan, Chatterjea, Reich,  Bianchini, Khan and many other contractions can be deduced.
Along with these, we show that the fixed point problems related to these two contractions are well-posed and possess limit shadowing property.  We give few examples  to support the validity of our results.

Before going to our main work, we first recollect a few result and definitions, which will be useful in our next sections.
At first, we recall the following result due to  Akram et. al\cite{AZS08}.
 
\begin{theorem} {\bf cf. (\cite[p. 29, Theorem 5]{AZS08}).}\label{dt1} 
Let a self map $T$ on a complete metric space $(X, d)$ satisfies the condition:
$$d(Tu, Tv)\le \alpha(d(u,v), d(u, Tu), d(v, Tv))$$
for all $u, v\in X$ and some $\alpha\in A$, where $A$  is the collection of all functions $\alpha:\mathbb{R}^3_+\to \mathbb{R}_+$ satisfying
\begin{itemize}
\item[{(i)}] $\alpha$ is continuous on the set $\mathbb{R}_+$ (with respect to the Euclidean metric on $\mathbb{R}^3$).
\item[{(ii)}] $a\le kb$ for some $k\in [0,1)$ whenever $a\le \alpha(a, b, b)$ or $a\le \alpha(b, a, b)$ or $a\le \alpha (b, b, a)$ for all $a, b$.
\end{itemize} 
Then $T$ is a Picard operator.
\end{theorem}
Next, we recall the well-posedness and limit shadowing property of a mapping.
\begin{definition} \label{d10}
Let $T$ be a self map defined on a metric space $(X, d)$. Then the fixed point problem concerning $T$ is known as well-posed if the followings hold:
\begin{itemize}
\item[{(i)}] $T$ has a unique fixed point $p\in X$;
\item[{(ii)}] for any sequence $\{u_n\}$ in $X$ with $\displaystyle\mathop{\lim_{n\to \infty}d(u_n, Tu_n)=0}$, we have 
$$\displaystyle\mathop{\lim_{n\to \infty}d(u_n, p)=0}.$$
\end{itemize}
\end{definition}
\begin{definition}\label{d12}
Let $T$ be a self-map defined on a metric space $(X, d)$. Then the fixed point problem involving $T$ is said to possess limit shadowing property in $X$ if for any sequence $\{u_n\}$ in $X$ such that $\displaystyle\mathop{\lim_{n\to \infty}d(u_n,Tu_n)=0}$, we have $p\in X$ with $\displaystyle\mathop{\lim_{n\to \infty}d(T^np,u_n)=0}$.
\end{definition}

\section{Enriched $\mathcal{A}$-contraction and enriched $\mathcal{A}'$-contraction}
In this section we give the formal definitions of enriched $\mathcal{A}$-contraction and enriched $\mathcal{A}'$-contraction mappings. To do this, we need two family of mappings which satisfy certain properties. Let $\mathcal{A}$ be the collection of all mappings $f:\mathbb{R}_+^3\to \mathbb{R}$ satisfying the following conditions:
\begin{itemize}
\item[{($\mathcal{A}_1$)}] $f$ is continuous;
\item[{($\mathcal{A}_2$)}] if $r\le f(s,r,s)$ or $r\le f(r, s, s)$, then there exists $k\in [0, 1)$ such that $r\le ks$;
\item[{($\mathcal{A}_3$)}] for $\lambda>0$ and for all $r,s,t\in \mathbb{R}_+$, $\lambda f(r,s,t)\le f(\lambda r, \lambda s, \lambda t)$.
\end{itemize}
 
Let $\mathcal{A}'$ be the collection of all mappings $f:\mathbb{R}_+^3\to \mathbb{R}$ satisfying the following conditions:
\begin{itemize}
\item[{($\mathcal{A}'_1$)}] $f$ is continuous;
\item[{($\mathcal{A}'_2$)}] if $r\le f(r, s, s)$ or $r\le f(s, s, r)$, then there exists $k\in [0, 1)$ such that $r\le ks$;
\item[{($\mathcal{A}'_3$)}] for $\lambda>0$ and for all $r,s,t\in \mathbb{R}_+$, $\lambda f(r,s,t)\le f(\lambda r, \lambda s, \lambda t)$;
\item[{($\mathcal{A}'_4$)}] if $t\le t_1$, then $f(r,s,t)\le f(r,s,t_1)$;
\item[{($\mathcal{A}'_5$)}] if $r \le f(s,0,  r + s)$, then $r \le ks$ for some $k\in [0,1)$;
\item[{($\mathcal{A}'_6$)}] if $r\le f(r, r, r)$, then $r=0$.
\end{itemize}

Below we present few example of mappings $f$ belonging to the class $\mathcal{A}$:  
\begin{itemize}
\item [(i)] $f(r,s,t)=\alpha(s+t)$, where $0\leq \alpha<\frac{1}{2}$;
\item [(ii)] $f(r,s,t)=\alpha \max\{s,t\}$, where $0\leq \alpha<1$;
\item [(iii)] $f(r,s,t)=\alpha\max\{r,s,t\}$, where $0\leq \alpha<1$;
\item [(iv)] $f(r,s,t)=\alpha_1r+\alpha_2s+\alpha_3t$, where $0\leq \alpha_1,\alpha_2,\alpha_3<1$ and $\alpha_1+\alpha_2+\alpha_3<1$.
\end{itemize}
The following mappings $f$ belong to the class $\mathcal{A}'$: 
\begin{itemize}
\item [(i)] $f(r,s,t)=\alpha(s+t)$, where $0\leq\alpha<\frac{1}{2}$;
\item [(ii)] $f(r,s,t)=\alpha(r+s+t)$, where $0\leq\alpha<1$;
\item [(iii)] $f(r,s,t)=\alpha\max\{s,t\}$, where $0\leq\alpha<1$.
\end{itemize}
Next, we define enriched $\mathcal{A}$-contraction and enriched $\mathcal{A}'$-contractions in the following way.
\begin{definition} \label{d01}
Let $(X, \|\cdot\|)$ be a real Banach space. Let $T:X\to X$ be a mapping such that there exists an $f\in \mathcal{A}$ with
$$\|b(u-v)+Tu-Tv\|\le f((b+1)\|u-v\|, \|u-Tu\|, \|v-Tv\|)\eqno(1)$$
for all $u, v\in X$ with $u\neq v$ and $b\in [0, \infty)$.
Then $T$ is said to be an enriched $\mathcal{A}$-contraction.
\end{definition}
\begin{definition} \label{d2}
Let $(X, \|\cdot\|)$ be a real Banach space. Let $T:X\to X$ be a mapping such that there exists an $f\in \mathcal{A}'$ with
$$\|b(u-v)+Tu-Tv\|\le f((b+1)\|u-v\|, \|(b+1)(u-v)+v-Tv\|, \|(b+1)(v-u)+u-Tu\|)\eqno(2)$$
for all $u, v\in X$ with $u\neq v$ and $b\in [0, \infty)$.
Then $T$ is said to be an enriched $\mathcal{A}'$-contraction.
\end{definition}
Next, we give some examples of the above two types of contractions.
\begin{example}
Let $X=\mathbb{R}$ and take the usual norm on $X$. Consider the mappings $T_1,T_2:X\to X$ defined by $$T_1u=-2u$$ for all $x\in X$ and $$T_2u=\begin{cases}
u+16,~~$if$~~u\in [1,2]\\
16,~~$if$~~u\notin [1,2].
\end{cases}$$
Then $T_1$ is an enriched $\mathcal{A}$-contraction for $b=\frac{5}{4}$ and $f(r,s,t)=\frac{s+t}{3}$, and $T_2$ is an enriched $\mathcal{A}'$-contraction for $b=\frac{1}{3}$ and $f(r,s,t)=\frac{1}{3}r+\frac{1}{4}s+\frac{1}{4}t.$ 
\end{example}

Next, we present an example which is not an enriched $\mathcal{A}$-contraction.
\begin{example}\label{e3}
Let $X=\mathbb{R}$ be the Banach space equipped with usual norm. Define 
$$ Tu=\left\{%
\begin{array}{ll}
    1+u &\hbox{if $u\ge 0$}\\
    0 & \hbox{if $u<0$}
\end{array}%
\right.$$ for all $u\in X$.

We claim that $T$ is not an enriched $\mathcal{A}$-contraction on $X$. If possible, let $T$ be an enriched $\mathcal{A}$-contraction on $X$. Then there exists $b\in [0, \infty)$ such that
$$\|b(u-v)+Tu-Tv\|\le f((b+1)\|u-v\|, \|u-Tu\|, \|v-Tv\|)$$
for all $u, v\in X$.

Let us take $u, v\in X$ be such that $u, v\ge 0$. Then
\begin{align*}
&\|b(u-v)+Tu-Tv\|\le f((b+1)\|u-v\|, \|u-Tu\|, \|v-Tv\|)\\
&\Longrightarrow \|(b+1)(u-v)\|\le f((b+1)\|u-v\|, 1, 1).
\end{align*}
Then there exists $k\in [0, \infty)$ such that
\begin{align*}
&|(b+1)| |u-v|\le k\cdot 1\\
&\Longrightarrow |u-v|\le \frac{k}{b+1}
\end{align*}
for all $u, v\ge 0$ with $u\neq v$, which is a contradiction.

Hence $T$ is not an enriched $\mathcal{A}$-contraction on $X$. Note that $T$ admits no fixed point in $X$.
\end{example}

\section{Main Results}\label{sec2}
 
In this section, we  present the following two results that ensure the existence of fixed point of the two contractions, mentioned here, if the domain of the mapping is a Banach space.

\begin{theorem}\label{t4}
Let $X$ be Banach space and $T$ be an enriched $\mathcal{A}$-contraction. Then $T$ has a unique fixed point in $X$ and there exists $\lambda \in (0, 1]$ such that the sequence $\{u_n\}$ defined by $u_{n+1}=(1-\lambda)u_n+\lambda Tu_n$, $n\ge 0$ converges to that fixed point, for any $u_o\in X$.
\end{theorem}
\begin{proof}
Let $b>0$. Set $\lambda=\frac{1}{b+1}>0$. Then the from $(1)$, we get
\begin{align*}
\|T_\lambda u-T_\lambda v\|&\le \lambda f\left(\frac{1}{\lambda}\|u-v\|, \|u-Tu\|, \|v-Tv\|\right)\\
&\le f(\|u-v\|, \|u-T_\lambda u\|, \|v-T_\lambda v\|)\hspace{12em}(3)
\end{align*}
for all $u, v\in X$ with $u\neq v$.

Let $u_0\in X$ be arbitrary. Define $u_{n}=T_\lambda^nu_0$ for all $n\ge 1$. Let us put $u=u_n$ and $v=u_{n-1}$ in $(2)$. Then we have
$$\|u_{n+1}-u_n\|\le f(\|u_n-u_{n-1}\|, \|u_n-u_{n+1}\|, \|u_{n-1}-u_{n}\|)$$
which implies that
$$\|u_{n+1}-u_n\|\le k\|u_n-u_{n-1}\|$$
for some $k\in [0, 1)$. From this we get
$$\|u_{n+1}-u_n\|\le k^n\|u_1-u_{0}\|.$$

Now for all $m,n\ge 1$, we have
\begin{align*}
\|u_{n+m}-u_n\|&\le \|u_{n+m}-u_{n+m-1}\|+\|u_{n+m-1}-u_{n+m-2}\|+\cdots+\|u_{n+1}-u_n\|\\
&\le (k_{n+m-1}+k^{n+m-2}+\cdots +k^n)\|u_1-u_0\|\\
&=k^n \frac{1-k^m}{1-k}\|u_1-u_0\| 
\end{align*}
which implies that $\|u_{n+m}-u_n\|\to 0$ as $m, n\to \infty$. Hence $\{u_n\}$ is a Cauchy sequence in $X$ and hence there exists an element $p\in X$ such that $u_n\to p$ as $n\to \infty$.

Now
$$\|T_\lambda p-u_{n+1}\|\le f(\|p-u_n\|, \|p-T_\lambda p\|, \|u_n-u_{n+1}\|).$$
Taking limit as $n\to \infty$, we get
$$\|T_\lambda p-p\|\le f(\|p-p\|, \|p-T_\lambda p\|, \|p-p\|)$$
which implies that
$$\|T_\lambda p-p\|\le k\cdot \|p-p\|=0.$$ 
This implies that $T_\lambda p=p$ and consequently $Tp=p$. Therefore, $p$ is a fixed point of $T$. 

Let $q\in X$ be another fixed point of $T$ and consequently a fixed point of $T_\lambda$. Then
\begin{align*}
\|p-q\|&=\|T_\lambda p-T_\lambda q\|\\
&\le f(\|p-q\|, \|p-T_\lambda p\|, \|q-T_\lambda q\|)\\
&=f(\|p-q\|, \|p-p\|, \|q-q\|)\\
&=f(\|p-q\|, 0, 0).
\end{align*} 
Therefore there exists a $k\in [0, 1)$ such that $\|p-q\|\le k\cdot 0=0$. Hence $p=q$.

If $b=0$, then $(1)$ reduces to the form
$$\|T u-T v\|\le f(\|u-v\|, \|u-T u\|, \|v-T v\|)$$
   for all $u,v\in X$ with $u\neq v$. Hence the result follows from Theorem \ref{dt1}.
\end{proof}
\begin{theorem}\label{t5}
Let $X$ be Banach space and $T$ be an enriched $\mathcal{A}'$-contraction. Then $T$ has a unique fixed point in $X$ and there exists $\lambda \in (0, 1]$ such that the sequence $\{u_n\}$ defined by $u_{n+1}=(1-\lambda)u_n+\lambda Tu_n$, $n\ge 0$ converges to that fixed point, for any $u_0\in X$.
\end{theorem}
\begin{proof}
Let $b>0$. Set $\lambda=\frac{1}{b+1}$ so that $0<\lambda<1$. Now note that
$$u-T_\lambda v=(u-v)+\lambda(v-Tv)\mbox{ and } v-T_\lambda u=(v-u)+\lambda(u-Tu).$$
Then the from $(2)$, we get
\begin{align*}
\|T_\lambda u-T_\lambda v\|&\le \lambda f\left(\frac{1}{\lambda}\|u-v\|, \left\|\frac{1}{\lambda}(u-v)+v-Tv\right\|, \left\|\frac{1}{\lambda}(v-u)+u-Tu\right\|\right)\\
&\le f(\|u-v\|, \|(u-v)+\lambda(v-T v)\|, \|(v-u)+\lambda(u-T u)\|)\\
&=f(\|u-v\|, \|u-T_\lambda v\|, \|v-T_\lambda u\|)\hspace{12em}(4)
\end{align*}
for all $u, v\in X$ with $u\neq v$.

Let $u_0\in X$ be arbitrary. Define $u_{n}=T_\lambda^nu_0$ for all $n\ge 1$. Let us put $u=u_n$ and $v=u_{n-1}$ in $(2)$. Then we have
\begin{align*}
\|u_{n+1}-u_n\|&\le f(\|u_n-u_{n-1}\|, \|u_n-u_{n}\|, \|u_{n-1}-u_{n+1}\|)\\
&\le f(\|u_n-u_{n-1}\|, \|u_n-u_{n}\|, \|u_{n-1}-u_n\|+\|u_n-u_{n+1}\|)
\end{align*}
which implies that
$$\|u_{n+1}-u_n\|\le k\|u_n-u_{n-1}\|$$
for some $k\in [0, 1)$. From this we get
$$\|u_{n+1}-u_n\|\le k^n\|u_1-u_{0}\|.$$

Now for all $m,n\ge 1$, we have
\begin{align*}
\|u_{n+m}-u_n\|&\le \|u_{n+m}-u_{n+m-1}\|+\|u_{n+m-1}-u_{n+m-2}\|+\cdots+\|u_{n+1}-u_n\|\\
&\le (k_{n+m-1}+k^{n+m-2}+\cdots +k^n)\|u_1-u_0\|\\
&=k^n \frac{1-k^m}{1-k}\|u_1-u_0\| 
\end{align*}
which implies that $\|u_{n+m}-u_n\|\to 0$ as $m, n\to \infty$. Hence $\{u_n\}$ is a Cauchy sequence in $X$ and hence there exists an element $p\in X$ such that $u_n\to p$ as $n\to \infty$.

Now
$$\|T_\lambda p-u_{n+1}\|\le f(\|p-u_n\|, \|p-u_{n+1}\|, \|u_n-T_\lambda p\|).$$
Taking limit as $n\to \infty$, we get
$$\|T_\lambda p-p\|\le f(\|p-p\|, \|p-p\|, \|p-T_\lambda p\|)$$
which implies that
$$\|T_\lambda p-p\|\le k\cdot \|p-p\|=0$$ for some $k\in [0, 1)$. 

Hence $T_\lambda p=p$ and consequently, $Tp=p$. Thus $p$ is a fixed point of $T$. 

Let $q\in X$ be another fixed point of $T$ and hence a fixed point of $T_\lambda$. Then
\begin{align*}
\|p-q\|&=\|T_\lambda p-T_\lambda q\|\\
&\le f(\|p-q\|, \|p-T_\lambda q\|, \|q-T_\lambda p\|)\\
&=f(\|p-q\|, \|p-q\|, \|q-p\|)
\end{align*} 
which implies that $\|p-q\|=0$. Hence $p=q$.

If $b=0$, then $(2)$ reduces to the form
$$\|T u-T v\|\le f(\|u-v\|, \|u-T v\|, \|v-T u\|)$$
for all $u, v\in X$ with $u\neq v$. 

Let $u_0\in X$. Define the sequence $\{u_n\}$ by $u_n=T^nu_0$ for all $n\ge 1$. Then just replacing $T_\lambda$ by $T$ in the above proof, the result follows.  
\end{proof}
We now note an important observation  about the above two results, which is presented in the following remark.
\begin{remark}
In most of the contractions for Banach space valued mappings, we see that if the domain of the mapping is a closed subset (not necessarily a subspace) of a Banach space, then the mapping possesses a fixed point. But this fact is not true in case of enriched  contractions. More precisely, in the above two theorems, if the domain of the mapping $T$ is a closed subset of a Banach space, then $T$ may not acquire a fixed point. To support this, we consider the following example.
\begin{example}
Choose $X=\mathbb{R}$ with usual norm and take $C=(-\infty,-1]\cup [1,\infty)$. Define $T:C \to C$ by $Tu=-2u$ for all $u\in C$. Then $T$ is an enriched $\mathcal{A}$-contraction for $b=\frac{5}{4}$ and $f(r,s,t)=\frac{s+t}{3}$ but $T$ is fixed point free.
\end{example}
\end{remark}
In view of Remark~3.3., we pose the following open question:
\begin{open question} 
Does there exist a self map defined on a subspace $($not necessarily closed$)$ of a Banach space satisfying enriched $\mathcal{A}$-contraction $($resp. enriched $\mathcal{A}'$-contraction$)$ having a fixed point?
\end{open question}

By choosing particular $f$ in above two theorems, we have the following remarks.
\begin{remark}
If we choose $f(r,s,t)=\alpha r$, where $0\leq \alpha<1$; $f(r,s,t)=\alpha(s+t)$, where $0\leq \alpha<\frac{1}{2}$; $f(r,s,t)=\alpha_1r+\alpha_2s+\alpha_3t$, where $0\leq \alpha_1,\alpha_2,\alpha_3<1$ and $\alpha_1+\alpha_2+\alpha_3<1$; $f(r,s,t)=\alpha \max\{s,t\}$, where $0\leq \alpha<1$; $f(r,s,t)=\alpha \sqrt{st}$, where $0\leq \alpha<1$ in Theorem \ref{t4}, then we get the enriched versions of the contractions of Banach \cite{B22}, Kannan \cite {Rk68}, Reich \cite{R3}, Bianchini \cite{B10} and Khan \cite{K78} respectively. 
\end{remark}
\begin{remark}
If we choose $f(r,s,t)=\alpha(s+t)$, where $0\leq \alpha<\frac{1}{2}$ in Theorem \ref{t5}, then we get the enriched version of the  Chatterjea contraction \cite {Ch}. Also, for  $f(r,s,t)=\alpha\max\{s,t\}$, where $0\leq \alpha<1$; $f(r,s,t)=\alpha_1r+\alpha_2s+\alpha_3t$, where $0\leq \alpha_1,\alpha_2,\alpha_3<1$ and $\alpha_1+\alpha_2+\alpha_3<1$ in Theorem \ref{t5}, the enriched versions of the following two contractions  respectively, can be acquired:
\begin{itemize}
\item $\|Tu-Tv\|\leq \alpha \max\{\|u-Tv\|,\|v-Tu\|\}$, $0\leq \alpha<1$;
\item $\|Tu-Tv\|\leq \alpha_1\|u-v\|+\alpha_2 \|u-Tv\|+\alpha_3\|v-Tu\|$,  $0\leq \alpha_1,\alpha_2,\alpha_3<1$ and $\alpha_1+\alpha_2+\alpha_3<1$. 
\end{itemize}
\end{remark}
\begin{remark}
One can perceive from the above remarks that Berinde and  P\u{a}curar's results regarding existence of fixed point of different types of enriched contractions \cite{BP20, BP1, BP2} are particular cases of our results.
\end{remark}
For $b=0$ and some particular $f$ in Theorem \ref{t4} and \ref{t5}, we obtain some classical fixed point results, which is presented in the next remarks.
\begin{remark}
In particular for $b=0$ and for $f(r,s,t)=\alpha(s+t)$, where $0\leq \alpha<\frac{1}{2}$; $f(r,s,t)=\alpha_1r+\alpha_2s+\alpha_3t$, where $0\leq \alpha_1,\alpha_2,\alpha_3<1$ and $\alpha_1+\alpha_2+\alpha_3<1$; $f(r,s,t)=\alpha \max\{s,t\}$, where $0\leq \alpha<1$; $f(r,s,t)=\alpha \sqrt{st}$, where $0\leq \alpha<1$ in Theorem \ref{t4}, we can obtain the classical fixed point results of Kannan \cite {Rk68}, Reich \cite{R3} and Bianchini \cite{B10}, Khan \cite{K78} respectively as corollaries.
\end{remark}
\begin{remark}
For $b=0$ and $f(r,s,t)=\alpha(s+t)$, where $0\leq \alpha<\frac{1}{2}$ in Theorem \ref{t5},  the classical fixed point theorem of  Chatterjea \cite {Ch} can be deduced as a corollary. Also for $b=0$ and $f(r,s,t)=\alpha\max\{s,t\}$, where $0\leq \alpha<1$; $f(r,s,t)=\alpha_1r+\alpha_2s+\alpha_3t$, where $0\leq \alpha_1,\alpha_2,\alpha_3<1$ and $\alpha_1+\alpha_2+\alpha_3<1$ in Theorem \ref{t5},  two fixed point results concerning the following contractions respectively, can be deduced:
\begin{itemize}
\item $\|Tu-Tv\|\leq \alpha \max\{\|u-Tv\|,\|v-Tu\|\}$, $0\leq \alpha<1$;
\item $\|Tu-Tv\|\leq \alpha_1\|u-v\|+\alpha_2 \|u-Tv\|+\alpha_3\|v-Tu\|$,  $0\leq \alpha_1,\alpha_2,\alpha_3<1$ and $\alpha_1+\alpha_2+\alpha_3<1$. 
\end{itemize}
\end{remark}

We now take a couple of  examples in support of the  Theorem \ref{t4}.
\begin{example}\label{e6}
Let us consider the Banach space $(X,\Vert \cdot \Vert)$, where $X=C[0,1]$ and $\Vert \cdot \Vert$ is the sup norm. Next, we define a mapping $T:X\to X$ by $(Tu)(t)=-2u(t)$ for all $u\in X$ and $t\in [0, 1]$. We choose $f\in \mathcal{A}$, defined by $f(r,s,t)=\frac{1}{3}(s+t)$ and $b=\frac{5}{4}$. Then for any $u,v\in X$, we have 
\begin{align*}
f((b+1)\|u-v\|, \|u-Tu\|, \|v-Tv\|)&=\frac{1}{3}\Big(\|u-Tu\|+\|v-Tv\|\Big)\\
&=\|u\|+\|v\|,
\end{align*}
and 
\begin{align*}
\|b(u-v)+Tu-Tv\|&=\|b(u-v)-2u+2v\|\\
&=\vert b-2 \vert \Vert u-v \Vert=\frac{3}{4}\Vert u-v \Vert.
\end{align*}
Therefore,  $$\|b(u-v)+Tu-Tv\|\leq f((b+1)\|u-v\|, \|u-Tu\|, \|v-Tv\|)$$ holds for all $u,v\in X$. So $T$ is an enriched $\mathcal{A}$-contraction.  So by Theorem \ref{t4}, $T$ has a unique fixed point. Note that $u\in X$ defined by $u(t)=0$ for all $t\in [0,1]$, is the unique fixed point of $T$
\end{example}
\begin{example}\label{e7}
Let $X=\mathbb{R}$ be the Banach space equipped with usual norm. Define $T:X\to X$ by $Tu=6-u$ for all $u\in X$.

Let us consider the mapping $f:\mathbb{R}_+^3\to \mathbb{R}_+$ defined by $f(r, s, t)=\frac{1}{6}\max\{s,t\}$ for $r, s, t\in \mathbb{R}_+$. Then $f\in \mathcal{A}$. Let $b=1$. Then we have 
$$\|b(u-v)+Tu-Tv\|=0$$ and 
$$f((b+1)\|u-v\|, \|u-Tu\|, \|v-Tv\|)=\frac{1}{6}\max\{\|2u-6\|,\|2v-6\|,$$ which shows that 
$$\|b(u-v)+Tu-Tv\|\leq f((b+1)\|u-v\|, \|u-Tu\|, \|v-Tv\|)$$
for all $u, v\in X$. Therefore, $T$ is an enriched $\mathcal{A}$-contraction on $X$. It is to be noticed that $3$ is the unique fixed point of $T$.
\end{example}
Next, we present another couple of examples, which validate Theorem \ref{t5}.
\begin{example}\label{e8}
Let us take the Banach space $X=C[0,\frac{1}{4}]$, equipped  with sup norm. We define a mapping $T:X\to X$ by $(Tu)(t)=tu(t)$ for all $u\in X$ and for all $t\in [0, \frac{1}{4}]$. We choose $b=\frac{1}{4}$ and $f\in \mathcal{A}'$, defined by $f(r,s,t)=\frac{9}{20}\max\{r,s,t\}$. Then for any $u,v \in X$, we have
\begin{align*}
&\|b(u-v)+Tu-Tv\|\\
&=\sup_{t\in [0, \frac{1}{4}]} |bu(t)-bv(t)+tu(t)-tv(t)|\\
&\leq \sup_{t\in [0, \frac{1}{4}]}|b+t| \sup_{t\in [0, \frac{1}{4}]}|u(t)-v(t)|\\
&=\frac{1}{2}\|u-v\|\\
&\leq \frac{9}{20} (b+1)\|u-v\|\\
&\leq  f((b+1)\|u-v\|, \|(b+1)(u-v)+v-Tv\|, \|(b+1)(v-u)+u-Tu\|).
\end{align*}
Thus $T$ is an enriched $\mathcal{A}'$-contraction. So by  Theorem~\ref{t5}, it follows that $T$ possesses a unique fixed point in $X$ and note the fixed point is $u'$, where $u'(t)=0$ for all $t\in [0,\frac{1}{4}]$.
\end{example}

\begin{example}\label{e9}
Let $X=\mathbb{R}$ be the Banach space equipped with usual norm. Define $T:X\to X$ by $Tu=2-u$ for all $u\in X$.

Let us consider the mapping $f:\mathbb{R}_+^3\to \mathbb{R}_+$ defined by $f(r, s,t)=\frac{s+t}{6}$ for $u, v, w\in \mathbb{R}_+$. Then $f\in \mathcal{A}'$. Let us choose $b=1$. Then we have 
$$\|b(u-v)+Tu-Tv\|=0$$ 
and 
$$f((b+1)\|u-v\|, \|(b+1)(u-v)+v-Tv\|, \|(b+1)(v-u)+u-Tu\|)=\frac{|u-1|+|v-1|}{3}.$$ 
So $$\|b(u-v)+Tu-Tv\|\leq f((b+1)\|u-v\|, \|(b+1)(u-v)+v-Tv\|, \|(b+1)(v-u)+u-Tu\|)$$       
for all $u, v\in X$.
Therefore, $T$ is an enriched $\mathcal{A}'$-contraction on $X$. It is to be noticed that $1$ is the unique fixed point of $T$.
\end{example}

Finally, we study well-posedness and limit shadowing property of fixed point problem for both types of contractions defined in the present paper.

\begin{theorem} \label{t10}
Let $X$ be Banach space and $T$ be an enriched $\mathcal{A}$-contraction. Then the fixed point problem is well posed.
\end{theorem}
\begin{proof}
Theorem~\ref{t4} ensures us that $T$ possesses a unique fixed point $p$, say.

It has been shown in Theorem~\ref{t4} that if $b\in (0, \infty)$, then the contraction condition
$$\|b(u-v)+Tu-Tv\|\le f((b+1)\|u-v\|, \|u-Tu\|, \|v-Tv\|)$$
for all $u, v\in X$ with $u\neq v$, can be reduced to the form
$$\|T_\lambda u-T_\lambda v\|\le f(\|u-v\|, \|u-T_\lambda u\|, \|v-T_\lambda v\|)$$
for some $\lambda>0$. In this case it is to be noted that $\|u-T_\lambda u\|=\lambda \|u-Tu\|$ for all $u\in X$.

Therefore, $\displaystyle\mathop{\lim_{n\to \infty}\|u_n-Tu_n\|=0\Longleftrightarrow \lim_{n\to \infty}\|u_n-T_\lambda u_n\|=0}$.

Now let $\{u_n\}$ be a sequence in $X$ such that $\displaystyle\mathop{\lim_{n\to \infty}\|u_n-T_\lambda u_n\|=0}$. Then
\begin{align*}
\|u_n-p\|&\le \|u_n-T_\lambda u_n\|+\|T_\lambda u_n-p\|\\
& =\|u_n-T_\lambda u_n\|+\|T_\lambda u_n-Tp\|\\
&\le \|u_n-T_\lambda u_n\|+f(\|u_n-p\|, \|u_n-T_\lambda u_n\|, \|p-T_\lambda p\|).
\end{align*}
Taking limit as $n\to \infty$, we get
$$\lim_{n\to \infty} \|u_n-p\|\le f(\lim_{n\to \infty} \|u_n-p\|, 0,0).$$
Therefore there exist $k\in [0, 1)$ such that $\displaystyle\mathop{\lim_{n\to \infty}\|u_n-p\|\le k\cdot 0}$ which implies that
$\displaystyle\mathop{\lim_{n\to \infty}\|u_n-p\|=0}$.

Now, if $b=0$, then the contraction condition reduces to
$$\|Tu-Tv\|\le f(\|u-v\|, \|u-Tu\|, \|v-Tv\|)$$
for all $u, v\in X$ with $u\neq v$.

In a similar way, it can be shown that $\displaystyle\mathop{\lim_{n\to \infty}\|u_n-p\|=0}$.
Hence the result follows.
\end{proof}

\begin{theorem} \label{t11}
Let $X$ be Banach space and $T$ be an enriched $\mathcal{A}'$-contraction. Then the fixed point problem is well posed.
\end{theorem}
\begin{proof}
Theorem~\ref{t5} ensures us that $T$ possesses a unique fixed point $p$, say.

It has been shown in Theorem~\ref{t5} that if $b\in (0, \infty)$, then the contraction condition
$$\|b(u-v)+Tu-Tv\|\le f((b+1)\|u-v\|, \|(b+1)(u-v)+v-Tv\|, \|(b+1)(v-u)+u-Tu\|)$$
for all $u, v\in X$ with $u\neq v$, can be reduced to the form
$$\|T_\lambda u-T_\lambda v\|\le f(\|u-v\|, \|u-T_\lambda v\|, \|v-T_\lambda u\|)$$
for some $\lambda>0$. In this case it is to be noted that $\|u-T_\lambda v\|=\|(u-v)+\lambda (v-Tv)\|$ and $\|v-T_\lambda u\|=\|(v-u)+\lambda(u-Tu)\|$ for all $u, v\in X$.

Therefore, for any sequence $\{u_n\}$ in $X$, we have 
$$\displaystyle\mathop{\lim_{n\to \infty}\|u_n-Tu_n\|=0\Longleftrightarrow \lim_{n\to \infty}\|u_n-T_\lambda u_n\|=0}.$$

Now let $\{u_n\}$ be a sequence in $X$ such that $\displaystyle\mathop{\lim_{n\to \infty}\|u_n-T_\lambda u_n\|=0}$. Then
$$\|u_n-p\|\le \|u_n-T_\lambda u_n\|+\|T_\lambda u_n-p\|.$$
Now,
\begin{align*}
\|T_\lambda u_n-p\|&=\|T_\lambda u_n-T_\lambda p\|\\
&\le f(\|u_n-p\|, \|u_n-T_\lambda p\|, \|p-T_\lambda u_n\|)\\
&=f(\|u_n-p\|, \|u_n-p\|, \|p-T_\lambda u_n\|)
\end{align*}
Then we get a $k\in [0, 1)$ such that $\|T_\lambda u_n-p\|\le k\cdot \|u_n-p\|$.

Therefore, we have
$$\|u_n-p\|\le \|u_n-T_\lambda u_n\|+k\cdot \|u_n-p\|$$
for some $k\in [0, 1)$. This implies that
$$\|u_n-p\|\le \frac{1}{1-k}\|u_n-T_\lambda u_n\|$$
and therefore by taking limit as $n\to \infty$, we have $\displaystyle\mathop{\lim_{n\to \infty}\|u_n-p\|=0}$.

Now, if $b=0$, then the contraction condition reduces to
$$\|Tu-Tv\|\le f(\|u-v\|, \|u-Tv\|, \|v-Tu\|)$$
for all $u, v\in X$ with $u\neq v$.

In a similar way, it can be shown that $\displaystyle\mathop{\lim_{n\to \infty}\|u_n-p\|=0}$.
Hence the result follows.
\end{proof}

\begin{theorem} \label{t12}
Let $X$ be Banach space and $T$ be an enriched $\mathcal{A}$-contraction $($resp. an enriched $\mathcal{A}'$-contraction$)$. Then the fixed point problem involving $T$ possesses limit shadowing property in $X$.
\end{theorem}
\begin{proof}
We have already shown in Theorem~\ref{t4} (resp. in Theorem~\ref{t5}) that $\displaystyle\mathop{\lim_{n\to \infty}\|u_n-p\|=0}$, where $p$ is the unique fixed point of $T$. Then for any $n\in \mathbb{N}$, $T^np=p$ and therefore
$$\lim_{n\to \infty} \|u_n-T^np\|=0.$$
That is
$$\lim_{n\to \infty} \|T^np-u_n\|=0.$$
Hence the fixed point problems has limit shadowing property in both the cases.
\end{proof}

\paragraph{\textbf{Acknowledgement}.}
The second named author would like to express his special thanks of gratitude to  CSIR, New Delhi, India, for their financial supports under the CSIR-SRF fellowship scheme (Award Number: $09/973(0018)/2017$-EMR-I).

\end{document}